\documentclass[12pt,reqno]{amsart}

\usepackage[utf8]{inputenc}

\usepackage{amsmath}
\usepackage{amsfonts}
\usepackage{amssymb}
\usepackage{verbatim}
\usepackage[left=2.9cm,right=2.9cm,top=3cm,bottom=3cm]{geometry}
\usepackage{enumerate}

\newcommand{\cP}{{\hat{\mathcal P}}}

\newcommand{\N}{\mathbb N}

\newtheorem{theorem}{Theorem}
\newtheorem{lemma}{Lemma}
\newtheorem{rem}{Remark}

\newtheorem{prop}{Proposition}

\newtheorem{example}{Example}
\newtheorem{definition}{Definition}
\newtheorem{corollary}{Corollary}

\title[construction of aggregation paradoxes]{Construction of aggregation paradoxes through Load-sharing dependence models}

\author{Emilio De Santis}
\address{University of Rome La Sapienza, Department of Mathematics
Piazzale Aldo Moro, 5, 00185, Rome, Italy}
\email{desantis@mat.uniroma1.it}

\author{Fabio Spizzichino}
\address{University of Rome La Sapienza
Piazzale Aldo Moro, 5, I-00185, Rome, Italy}
\email{fabio.spizzichino@fondazione.uniroma1.it}

\begin{document}

\begin{abstract}
We show that \textit{load-sharing} models (a very special class of multivariate probability models for non-negative random variables) can be used to obtain some basic results about a multivariate extension of stochastic precedence, and related paradoxes. Such results can be applied in some different fields. In particular, respective applications of them can be developed in the two different contexts of paradoxes, which arise in voting theory, and of the concept of \textit{signature} which arises in the frame of systems' reliability.


\end{abstract}

\maketitle
\section{Introduction}

\label{intro} In the field of reliability theory, the term
\textit{load-sharing model} is mostly used to designate a very special class
of multivariate survival models. Such models arise from a simplifying
condition of stochastic dependence among the lifetimes of units which start
working simultaneously, are embedded into a same environment, and are designed
to support one another (or to share a common load or a common resource).

In terms of this restricted class of multivariate models, we will obtain some
basic results about stochastic precedence, minima among non-negative random
variables, and related paradoxes. Such results can be applied in some
different fields, even far from probabilistic analysis of non-negative random
variables. In particular, direct applications can be developed in the study of
the paradoxes arising in voting theory.

Let $X_{1},\ldots,X_{m}$ be $m$ non-negative random variables\ defined on a
same probability space and satisfying the \textit{no-tie} assumption
$\mathbb{P}\left(  X_{i}\neq X_{j}\right)  =1$, for $i\neq j,i,j\in\lbrack
m]\equiv\{1,...,m\}$.

For any subset $A$ $\subseteq\lbrack m]$ and any $j\in A$ let $\alpha_{j}(A)$
be the probability that $X_{j}$ takes on the minimum value among all the other
variables $X_{i}$ with $i\in A$, as it will be formally defined by the formula
(\ref{alfas}) below. In some contexts, $\alpha_{j}(A)$ can also be seen as a
\textit{winning probability.}

We concentrate our attention on the family $\mathcal{A}_{\left(  m\right)
}\equiv\{\alpha_{j}(A):A\subset\lbrack m],\text{ }j\in A\}$.

When $A$ exactly contains two elements, $A:=\{i,j\}$ say, the inequality
$\alpha_{i}(A)\geq$ $\alpha_{j}(A)$ translates the condition that $X_{i}$
\textit{stochastically precedes} $X_{j}$. This notion has been considered
several times in the literature, possibly under different terminologies. In
the last few years, in particular, it has attracted interest for different
aspects and in different applied contexts; see e.g. references \cite{AS2000},
\cite{BSC2004}, \cite{DSS1}, \cite{FH2018} \cite{NR2010}. The same concept is
also related to a comparison of \textit{statistical preference}, see e.g. the
paper \cite{MontesEtAl20}, dealing with the frame of voting theory, and other
papers cited therein.

A relevant aspect of such a concept is the possibility of observing a
non-transitive behavior. Namely, for some triple of indexes $i , j,h$, the
inequalities
\[
\alpha_{j}\left(  \{i,j\}\right)  >\frac{1}{2},\alpha_{h}\left(
\{j,h\}\right)  >\frac{1}{2},\alpha_{i}\left(  \{h,i\}\right)  >\frac{1}{2}%
\]
can simultaneously hold. Classical references can be given where these topics
have been treated, also under different types of languages and notation. See
e.g. \cite{steinhausparadox}, \cite{trybula1969}.

Already at a first glance, non-transitivity of stochastic precedence reveals
to be analogous to non-transitivity in the frame of collective preferences in
the comparisons within pairs of candidates, which is demonstrated by the
Condorcet's paradox. As well known a very rich literature has been devoted to
this specific topic, starting from the studies developed by J. C. Borda and
M.J. Condorcet at the end of the $18^{th}$ Century. In relation with the
purposes of the present paper, a brief overview and a few helpful references
will be provided along Section \ref{secDiscussion}, below.

Also other types of aggregation paradoxes arise in voting theory, when
attention is focused on the elections with more than two candidates.
Correspondingly, analogous probabilistic aggregation paradoxes can arise in
the analysis of the family $\mathcal{A}_{\left(  m\right)  }$, when comparing
the winning probabilities $\alpha_{j}(A)$ for subsets containing more than two
elements. See e.g. \cite{Blyth72}, \cite{Saari(SIAM)}; see also \cite{DMS20}.

A strictly related context is also the one of \textit{intransitive dice} (see
e.g. \cite{Savage}, \cite{HMRZ20} and references therein) and of the classic
games among players, who respectively bet on the occurrence of different
events in a sequence of trials. In fact, some paradoxical phenomena can emerge
in such a context as well. Relevant special cases are the possible paradoxes
which are met in the analysis of times of first occurrence for different words
of fixed length in random sampling of letters from an alphabet. See e.g.
\cite{Li80}, \cite{GO1981}, \cite{BlomTh82}, \cite{DSS2}, and references
therein. This special field had initially motivated our own interest toward
these topics.

A common approach for studying and comparing paradoxes respectively arising in
voting theory and in the analysis of the family $\mathcal{A}_{\left(
m\right)  }$ for $m$-tuples of random variables has been worked out by Donald
G. Saari, at the end of the last century (\cite{Saari(dictionary)},
\cite{Saari(Borda)}, \cite{Saari(SIAM)}). An approach aiming to describe
ranking in voting theory by means of comparisons among random variables has
been developed in terms of stochastic orderings, see in particular
\cite{MontesEtAl20} and the references cited therein.

An important type of results had been shown by Saari aiming to actually
emphasize that all possible ranking paradoxes can be conceivably observed.
Furthermore, and equivalently, the same results can be translated into the
language related to ranking comparisons among random variables. Such results
can be seen as generalizations of the classical McGarvey's Theorem (see
\cite{McGarvey1953}), which shows the actual existence of arbitrarily
paradoxical situations related with an analysis restricted to the pairs of candidates.

As a main purpose of this paper we obtain, in terms of comparisons of
stochastic-precedence type among random variables, a result (Theorem
\ref{EpsilonEspliciti}) which leads to conclusions similar to those by Saari.
From a mathematical viewpoint, this result is however very different and is
obtained by exploiting characteristic features of load sharing models. In
particular it allows us to actually construct load sharing models which give
rise to any arbitrarily paradoxical situation.

Even if aiming to different purposes, our work has some aspects in common also
with the paper \cite{MontesEtAl20}, which is based on the concept of
stochastic precedence (or statistical preference) as well.

More detailed explanations about the meaning of our results can be provided
along the next section, after necessary definitions and preliminary remarks,
and in Section \ref{secDiscussion}.

More precisely, the plan of the paper is as follows.

In Section \ref{sec2} we present preliminary results about the winning
probabilities $\alpha_{j}(A)$ for $m$-tuples of lifetimes. In particular we
also consider the random indices $J_{1},\ldots,J_{m}$ defined through the
position
\[
J_{r}=i\Leftrightarrow X_{i}=X_{r:m},
\]
where $X_{1:m},\ldots,X_{m:m}$ denote order statistics, and point out how
the\textit{ }family $\mathcal{A}_{\left(  m\right)  }$ is determined by the
joint probability of $\left(  J_{1},\ldots,J_{m}\right)  $ over the space
$\Pi_{m}$ of the permutations of $\{1,...,m\}$, namely from the set of
probabilities
\[
p_{m}(j_{1},...,j_{m})=\mathbb{P}\left(  J_{1}=j_{1},\ldots,J_{m}%
=j_{m}\right)  ,
\]
for $\left(  j_{1},\ldots,j_{m}\right)  \in$ $\Pi_{m}$.

We also introduce some notation and definitions of necessary concepts such as
the one of \textit{ranking pattern}, a natural extension of the concept of
majority graph. A simple relation of \textit{concordance} between a ranking
pattern and a multivariate probability model for $\left(  X_{1},\ldots
,X_{m}\right)  $ will be defined.

In Section \ref{sec3}, we recall the definition of load-sharing models, which
can be seen as very special cases of absolutely continuous multivariate
distributions for $\left(  X_{1},\ldots,X_{m}\right)  $. In the absolutely
continuous case, a possible tool to describe a joint distribution is provided
by the set of the \textit{multivariate conditional hazard rate} (m.c.h.r.)
functions. Load-sharing models just arise by imposing a remarkably simple
condition on the form of such functions. Concerning with the latter functions,
we briefly provide basic definitions and some bibliographic references. We
then define special classes of load-sharing models and show related
properties, of interest for our purposes. In particular we consider an
extension of load-sharing to explicitly include an \textit{order-dependent}
load-sharing condition. In Theorem \ref{lemma2} we show that, for any
arbitrary probability distribution $\rho_{m}$ over the space $\Pi_{m}$, there
exist a load-sharing model for $\left(  X_{1},\ldots,X_{m}\right)  $ such
that
\[
\mathbb{P}\left(  J_{1}=j_{1},\ldots,J_{m}=j_{m}\right)  =\rho_{m}\left(
j_{1},\ldots,j_{m}\right)
\]
\ for $\left(  j_{1},\ldots,j_{m}\right)  \in$ $\Pi_{m}$. Such a load-sharing
model will in general be of the type order-dependent.

In terms of the definition of concordance introduced in Section \ref{sec2}, in
Section \ref{sec4} we introduce and formally state our results concerning
aggregation paradoxes. In particular we obtain a quantitative result (Theorem
\ref{EpsilonEspliciti}) which shows possible methods to explicitly construct
 load-sharing models concordant with arbitrarily assigned ranking patterns.
 The proof of Theorem \ref{EpsilonEspliciti}
will be given after presenting some technical preliminaries.

We conclude with Section \ref{secDiscussion}, where we briefly illustrate the
meaning of our results in the study of both paradoxes in voting
theory and of \textit{signature} in the context of systems'
reliability (see \cite{SamaniegoBook}), respectively.

\section{Notation, preliminaries, and problem assessment}

\label{sec2} In this section we dwell on definitions, notation, and
preliminary arguments, needed for introducing the results which will be
formally stated and proven in the sequel. Some further notation will be
introduced where needed in the next sections.

In the following we consider a fixed $m\in\N$ and denote by the symbol $[m]$
the set $\{1,2,\ldots,m\}$. For $m>1$, we denote by $\widehat{\mathcal{P}}(m)$
the family of subsets of $[m]$ having cardinality greater than one. The symbol
$|B|$ denotes, as usual, the cardinality of a set $B$.

We consider the non-negative random variables $X_{1},\dots,X_{m}$ and assume
the no-tie condition i.e., for $i\neq j,i,j\in\lbrack m]$,
\begin{equation}
\mathbb{P}(X_{i}\neq X_{j})=1. \label{no-tie}%
\end{equation}
Henceforth, $X_{1},\ldots,X_{m}$ will be sometimes referred to as the
\textit{lifetimes}. The symbols $X_{1:m},\ldots,X_{m:m}$ denote the
corresponding order statistics and $J_{1},\ldots,J_{m}$ are defined by the
position
\begin{equation}
J_{r}=i\Leftrightarrow X_{i}=X_{r:m} \label{randomJ}%
\end{equation}
for any $i,r\in\lbrack m]$. \ Relating to the $m$-tuple $\left(  X_{1}%
,\ldots,X_{m}\right)  $, we consider the family
\[
\mathcal{A}_{\left(  m\right)  }\equiv\{\alpha_{j}(A);A\in\widehat{\mathcal{P}%
}(m),j\in A\},
\]
where $\alpha_{j}(A)$ denotes the winning probability formally defined by
setting
\begin{equation}
\alpha_{j}(A):=\mathbb{P}(X_{j}=\min_{i\in A}X_{i}). \label{alfas}%
\end{equation}

For $k\in\lbrack m]$ and $j_{1}\neq,\ldots\neq j_{k}\in\lbrack m]$, we set
\begin{equation}
p_{k}^{(m)}(j_{1},\ldots,j_{k}):=\mathbb{P}(J_{1}=j_{1},\,J_{2}=j_{2}%
,\,\ldots,\,J_{k}=j_{k}). \label{nome}%
\end{equation}

Then we focus attention on the probabilities $\alpha_{j}\left(  A\right)  $ in
(\ref{alfas}), on the probabilities of permutations (\ref{nome}), which are
triggered by $\left(  X_{1},\dots,X_{m}\right)  $, and on corresponding
relations among them. Further aspects, concerning with non-transitivity and
other related paradoxes, will be then pointed out.

Looking for a formula to compute the probability $p_{k}^{(m)}(j_{1}%
,\ldots,j_{k})$ we now denote by $\mathcal{D}(B,k)$, where $B\subset\lbrack
m]$, the set of the $k$-permutations of the complementary set $B^{c}$, namely
the set of the ordered samples (without replacement) of size $k$ of elements
chosen outside $B$:
\begin{equation}
\mathcal{D}(B,k):=\{(i_{1},\ldots,i_{k}):i_{1},\ldots,i_{k}\not \in B\text{
and }i_{1}\neq i_{2}\neq\ldots\neq i_{k}\}.\label{insiemeD}%
\end{equation}
When $k=m-|B|$, $\mathcal{D}(B,k)$ is then the set of all the permutations of
the elements of $B^{c}$. In particular the set $\Pi_{m}$ of the permutations
of all the elements of $[m]$ becomes $\mathcal{D}(\emptyset,m)$, a form which
will be sometimes used for convenience later on. For $k=m$ in (4), we will
simply write $p_{m}$ in place of $p_{m}^{\left(  m\right)  }$.

The probability $p_{k}^{(m)}(j_{1},\ldots,j_{k})$ can be computed by the
formula
\begin{equation}
p_{k}^{(m)}(j_{1},\ldots,j_{k})=\sum_{(u_{k+1},\ldots,u_{m})\in\mathcal{D}%
(\{j_{1},\ldots,j_{k}\},m-k)}p_{m}(j_{1},\ldots,j_{k},u_{k+1},\ldots,u_{m}).
\label{somma1}%
\end{equation}

For $A=[m]$ one obviously has $\alpha_{j}(A)=\mathbb{P}(J_{1}=j)$. \ For
$A\subseteq\lbrack m]$, with $1<|A|<m$, $j\in A$, and partitioning the event
$\{X_{j}=\min_{i\in A}X_{i}\}$ in the form
\begin{equation}
\{X_{j}=\min_{i\in A}X_{i}\}=\{J_{1}=j\}\cup\left(  \bigcup_{k=1}^{m-\ell
}\{J_{1}\not \in A,\,J_{2}\not \in A,\,\ldots,\,J_{k}\not \in A,\,J_{k+1}%
=j\}\right)  , \label{partizione}%
\end{equation}
one can easily obtain the following claim, which will be frequently used below
when dealing with the probabilities $\alpha_{j}(A)$.

\begin{prop}
\label{prop1} Let $X_{1},\dots,X_{m}$ be non-negative random variables
satisfying the no-tie condition. Let $A\in\widehat{\mathcal{P}}(
m  )$ and $\ell=|A|$. Then for any $j\in A$ one has
\[
\alpha_{j}(A)=\mathbb{P}(X_{j}=\min_{i\in A}X_{i})=\mathbb{P}(J_{1}=j)+
\]
\begin{equation}
+\sum_{k=1}^{m-\ell}{\sum_{(i_{1},\ldots,i_{k})\in\mathcal{D}(
A,k)} }p_{k+1}^{(m)}(i_{1},\ldots,i_{k},j). \label{formula}%
\end{equation}
\end{prop}

As a consequence of (\ref{somma1}) and Proposition \ref{prop1} one immediately obtains

\begin{corollary}
\label{cor1} The probabilities $(\alpha_{j}(A):A\in\widehat{\mathcal{P}%
}(m),j\in A)$ are determined by the set of probabilities $\{p_{m}(j_{1},\ldots,j_{m}):(j_{1},\ldots,j_{m})\in\Pi_{m}\}$
\end{corollary}

\medskip

It immediately follows that also the conditional probabilities
\begin{equation}
\mathbb{P}(J_{k+1}=j_{k+1}|J_{1}=j_{1},...,J_{k}=j_{k})=\frac{p_{k+1}^{\left(
m\right)  }(j_{1},\ldots,j_{k+1})}{p_{k}^{\left(  m\right)  }(j_{1}%
,\ldots,j_{k})}, \label{ProbCondizPerJ}%
\end{equation}
are determined by the set formed by the probabilities $p_{m}(j_{1}%
,\ldots,j_{m}):(j_{1},\ldots,j_{m})\in\Pi_{m}$.

\bigskip

As mentioned, a central role in our work is played by comparisons of the type
\begin{equation}
\alpha_{i}(A)\geq\alpha_{j}(A)\text{, for }A\in\widehat{\mathcal{P}}(m),i,j\in
A. \label{MainComparison}%
\end{equation}

\bigskip

When $A\equiv\{i,j\}$ (with $i,j\in\lbrack m]$), the inequality appearing in
(\ref{MainComparison}) is just equivalent to the notion of \textit{stochastic
precedence} of $X_{i}$ with respect to $X_{j}$, as mentioned in the Introduction.

By limiting attention only to subsets $A\subset\lbrack m]$ with $|A|=2$, a
direct graph (or \textit{digraph}) $([m],E)$ can be associated to the family
$\mathcal{A}$, by defining $E\subset\lbrack m]\times\lbrack m]$ as the set of
oriented arcs such that
\[
(i,j)\in E\text{ if and only if }\alpha_{i}(\{i,j\})\geq\alpha_{j}(\{i,j\}).
\]
\label{RAS}

In the recent paper \cite{D21ranking} it has been proven for an arbitrary
digraph $G=([m],E)$ that one can build a Markov chain and suitably associated
hitting times $X_{1},\ldots,X_{m}$, so that the relations \eqref{RAS} give
rise to $G$. Such a construction is different from those introduced in the
present paper and can be useful in view of applications within fields
different from those considered here.

Borrowing from the language used in voting theory, $\left(  \left[  m\right]
,E\right)  $ can be called a \textit{majority graph}. Concerning the notion of
digraphs, and the related notions of asymmetric digraphs, complete digraphs,
tournaments, etc. we address the readers e.g. to \cite{BachmeierEtAl} for
explanations and more details from the viewpoint of voting theory.

More generally, for any fixed $A\subset\lbrack m]$ with $|A|\geq2$, we can
introduce a function $\sigma(A,\cdot)$, where $\sigma(A;\cdot):A\rightarrow
\{1,2,\ldots,|A|\}$, in order to define the ranking among the elements of $A$,
as induced by the probabilities $\alpha_{j}(A)$. More precisely we set
\[
\sigma(A,j)=\mathit{1}\text{, if }\alpha_{j}(A)=\max_{i\in A}\alpha_{i}(A),
\]

\[
\sigma(A,j)=\mathit{2}\text{, if }\sigma(A,j)\neq\mathit{1}\text{ and }%
\alpha_{j}(A)=\max_{i:\sigma_{A}(i)>\mathit{1}\text{,}}\alpha_{i}(A),
\]
and so on ... .

When, for a given $A\in\widehat{\mathcal{P}}(m)$, the values $\sigma(A,j)$,
$j\in A$, are all different then  $\sigma(A,\cdot):A\rightarrow
\{\mathit{1,2,\ldots,|A|}\}$ is a bijective function, namely $\sigma(A,\cdot)$
describes a permutation of the elements of $A$. Otherwise the image of
$\sigma(A,\cdot)$ is $[\bar{w}]=\{1,,\ldots,\bar{w}\}$ for some $\bar{w}<|A|$.

\bigskip

More generally, independently from the rankings induced by a family such as
$\mathcal{A}_{\left(  m\right)  }$, we will say that a mapping $\sigma
(A,\cdot):A\rightarrow\{\mathit{1,2,\ldots,|A|}\}$  is a \emph{ranking
function} when its image is $[\bar{w}]=\{1,,\ldots,\bar{w}\}$ for some
$\bar{w}\leq|A|$.

For $i,j\in A$, we say that $i$ \emph{precedes $j$ in $A$} according to the
ranking function $\sigma(A,\cdot):A\rightarrow\{\mathit{1,2,\ldots,|A|}\}$ if
and only if $\sigma(A,i)<\sigma(A,j).$ We say that two elements are\emph{
equivalent in $A$} with respect to $\sigma_{A}$ when $\sigma(A,i)=\sigma
(A,j)$. When some equivalence holds between two elements of $A$, namely when
$\bar{w}<|A|$, we say that $\sigma_{A}$ is a \emph{weak ranking function}.

\bigskip

Coming back to the random variables $X_{1},\ldots,X_{m}$ and extending
attention to all the subsets $A\in\widehat{\mathcal{P}}(m)$, the family
$\mathcal{A}_{\left(  m\right)  }$ gives rise to a family of ranking functions
$\{\sigma(A,\cdot)\}_{A\in\widehat{\mathcal{P}}(m)}$. \ In this respect, we
introduce the following notation and definition.

\begin{definition}
\label{def1} For $m\geq2$, a family of ranking functions
\begin{equation}
\boldsymbol{\sigma}\equiv\{\sigma(A,\cdot):A\in\widehat{\mathcal{P}}(m)\}.
\label{coda}%
\end{equation}
will be called a \emph{ranking pattern} over $[m]$. The collection of all the
ranking patterns over $[m]$ will be denoted by $\Sigma^{(m)}$. A ranking
\emph{pattern} containing some weak ranking functions will be called a
\emph{weak ranking pattern}. The collection of all the ranking patterns not
containing any weak ranking function will be denoted by $\hat{\Sigma}%
^{(m)}\subset\Sigma^{(m)}$.
\end{definition}

\begin{example}
\label{PriorOfDoReMi}Let $m=3$ and consider the weak ranking pattern
$\boldsymbol{\sigma}$, defined by the following positions:
\[
\sigma([3],3)=\mathit{1},\text{ }\sigma([3],1)=\mathit{2},\text{ }%
\sigma([3],2)=\mathit{3},
\]%
\[
\sigma(\{1,3\},3)=\mathit{1},\text{ }\sigma(\{1,3\},1)=\mathit{2}%
,\sigma(\{2,3\},3)=\mathit{1},\text{ }\sigma(\{2,3\},2)=\mathit{2},
\]%
\[
\sigma(\{1,2\},1)=\sigma(\{1,2\},2)=\mathit{1}.
\]
The latter equality says that 1 and 2 are equivalent in the set $A\equiv
\{1,2\}$ and the image set of the ranking function $\sigma(A,\cdot)$ is then
$\{\mathit{1}\}$ with $|\{\mathit{1}\}|<|A|$.
\end{example}

The concept of ranking pattern is a direct extension of the one of majority
graph. In other words, a ranking pattern can be seen as an ordinal variant of
a choice function.

\begin{definition}
\label{def2} We say that the ranking pattern $\boldsymbol{\sigma}%
\equiv\{\sigma(A,\cdot);A\in\widehat{\mathcal{P}}(m)\}$ and the\emph{ $m$%
-}tuple\emph{ $(X_{1},\dots,X_{m})$} are $p$\emph{-concordant} whenever, for
any $A\in\widehat{\mathcal{P}}(m)$, and $i,j\in A$ with $i\neq j$
\begin{equation}
\sigma(A,i)<\sigma(A,j)\Leftrightarrow\alpha_{i}(A)>\alpha_{j}(A),
\label{glialfa}%
\end{equation}%
\begin{equation}
\sigma(A,i)=\sigma(A,j)\Leftrightarrow\alpha_{i}(A)=\alpha_{j}(A).
\label{glialfauguali}%
\end{equation}

\end{definition}

We remind that the quantities $\sigma(A,i)$ are natural numbers belonging to
$[|A|]$, whereas the quantities $\alpha_{j}(A)$ are real numbers belonging to
$[0,1]$, and such that $\sum_{i\in A}\alpha_{i}(A)=1$. In view of the
definition of the quantities $\alpha_{j}(A)$, the inequality $\sigma
(A,i)<\sigma(A,j)$ in \eqref{glialfa} means $\mathbb{P}\left(  X_{i}%
=\min_{h\in A}X_{h}\right)  >\mathbb{P}\left(  X_{j}=\min_{h\in A}%
X_{h}\right)  $. Motivations for such a position will emerge in the sequel.

\begin{example}
\label{doremi}With $m=3$, consider non-negative random variables $X_{1}%
,X_{2},X_{3}$ such that
\[
p(1,2,3)=\frac{2}{18},\,\,\,p(2,1,3)=\frac{2}{18},\,\,\,p(3,2,1)=\frac{5}%
{18},
\]%
\[
p(3,1,2)=\frac{3}{18},\,\,\,p(2,3,1)=\frac{2}{18},\,\,\,p_{3}(1,3,2)=\frac
{4}{18}.
\]
Thus we have
\[
\alpha_{1}([3])=p(1,2,3)+p(1,3,2)=\frac{2}{18}+\frac{4}{18}=\frac{1}{3}%
\]
and
\[
\alpha_{1}(\{1,2\})=\frac{2}{18}+\frac{3}{18}+\frac{4}{18}=\frac{1}{2},
\]%
\[
\alpha_{1}(\{1,3\})=\frac{2}{18}+\frac{2}{18}+\frac{4}{18}=\frac{4}{9}.
\]
Similarly,
\[
\alpha_{2}([3])=\frac{2}{9},\alpha_{2}(\{1,2\})=\frac{1}{2},\alpha
_{2}(\{2,3\})=\frac{1}{3},
\]%
\[
\alpha_{3}([3])=\frac{4}{9},\alpha_{3}(\{1,3\})=\frac{5}{9},\alpha
_{3}(\{2,3\})=\frac{2}{3}.
\]
Then the triple $(X_{1},X_{2},X_{3})$ is p-concordant with the ranking pattern
$\boldsymbol{\sigma}$, which has been considered in Example
\ref{PriorOfDoReMi} above.

\end{example}

\medskip

\begin{rem}
\label{remone0}
Of course, a same ranking pattern $\boldsymbol{\sigma}$ can be concordant with
several different \emph{$m$}-tuples of random variables. Actually, the joint
distribution $\mathbb{P}_{X}$ of $\left(  X_{1},...,X_{m}\right)  $ determines
the distribution $\rho$ over $\Pi_{\left(  m\right)  }$ defined by the set of
probabilities
\begin{equation}
\{p_{m}(j_{1},\ldots,j_{m});(j_{1},\ldots,j_{m})\in\Pi_{\left(  m\right)
}\},\label{FamilyPofPermuttnsProbablties}%
\end{equation}
and $\rho$ determines $\boldsymbol{\sigma}$. On the other hand, one can find
different joint distributions determining a same distribution $\rho$.
\end{rem}

\begin{rem} \label{Otherparad}
As mentioned above, and as it is generally well-known, the phenomenon of
non-transitivity may arise in the analysis of the set of quantities of the
type $\alpha_{i}(\{i,j\})$, for $i\neq j\in\left[  m\right]  $ and of the
induced digraph $\left(  \left[  m\right]  ,E\right)  $. Different types
of paradoxes may also be encountered when considering a ranking pattern
$\boldsymbol{\sigma}$. In particular, for a set $A\in\widehat{\mathcal{P}}(m)$
and a triple of indexes $i,j,k$ $\in\left[  m\right]  $, with $i,j$ $\in A$,
$k\notin A$, it may simultaneously happen
\begin{equation}
\mathit{\sigma(A,i)}>\mathit{\sigma(A,j)},\mathit{\sigma(A\cup\{k\},i)}%
<\mathit{\sigma(A\cup\{k\},j)}.\label{StranezzeConMoltiCavalli}%
\end{equation}

\end{rem}

Looking in particular at (\ref{StranezzeConMoltiCavalli}), one can
conceptually imagine ranking patterns which are quite astonishing and
paradoxical, as in the next Example.

\begin{example}
\label{VeryParadox} Let us single out, say, the element $1\in\left[  m\right]
$ and fix attention on a ranking pattern $\boldsymbol{\sigma}\in
\Sigma^{\left(  m\right)  }$ satisfying the following conditions:%
\[
\sigma(A,i)=\mathit{1}%
\]
for $A=\{1,j\}$ with $j\neq1$, and
\[
\sigma(A,i)=\mathit{|A|}%
\]
whenever $A\subseteq\lbrack m]$ with $|A|>2$, $1\in A$. Namely the element $1$
precedes any other element $j\neq1$ when only two elements are compared, and
it is preceded by any other element when more than two elements are compared.
One can wonder if it is possible to actually find a probability distribution
for $\left(  X_{1},...,X_{m}\right)  $ which is $p$-concordant with
$\boldsymbol{\sigma}$.
\end{example}

Paradoxical patterns may induce the suspect that it is impossible to actually
find out probability models concordant with them. The question then arises
whether an arbitrarily given $\boldsymbol{\sigma}\equiv\{\sigma(A,i);A\in
\widehat{\mathcal{P}}(m)\}$ can really be concordant with some appropriate
model. One can furthermore wonder whether it is possible, in any case, to
explicitly construct one such model. In this respect the result in Section
\ref{construction} will show that, for any given ranking pattern
$\boldsymbol{\sigma}\in\hat{\Sigma}^{(m)}$, one can construct suitable
probability models which are actually $p$-concordant with it and which belong
to a restricted class of load-sharing models. It will also be shown in Section
\ref{sec3} that, for an arbitrarily given distribution $\rho$ over $\Pi_{m}$,
it is possible to identify probability distributions $\mathbb{P}_{\mathbf{X}}$
belonging to an enlarged class of load-sharing models and such that
$\mathbb{P}_{X}\mathbb{\rightarrow}$ $\rho$.

\section{Load-sharing models and related properties}
\label{sec3}

In this section attention will be limited to the case of lifetimes admitting
an absolutely continuous joint probability distribution. Such a joint
distribution can then be described by means of the corresponding joint density
function. An alternative description can also be made in terms of the family
of the \textit{Multivariate Conditional Hazard Rate} (m.c.h.r.) functions. The
two descriptions are in principle equivalent, from a purely analytical
viewpoint. However, they turn out to be respectively convenient to highlight
different features of stochastic dependence.

\begin{definition}
Let $X_{1},\ldots,X_{m}$ be non-negative random variables with an absolutely
continuous joint probability distribution. For fixed $k\in\lbrack m-1]$, let
$(i_{1},\ldots,i_{k},j)\in\mathcal{D}(\emptyset,k+1)$, and for an ordered
sequence $0<t_{1}<\cdots<t_{k}<t$, the multivariate conditional hazard rate
function $\lambda_{j}(t|i_{1},\ldots,i_{k};t_{1},\ldots,t_{k})$ is defined by
\[
\lambda_{j}(t|i_{1},\ldots,i_{k};t_{1},\ldots,t_{k}):=
\]%
\begin{equation}
\lim_{\Delta t\rightarrow0^{+}}\frac{1}{\Delta t}\mathbb{P}(X_{j}\leq t+\Delta
t|X_{i_{1}}=t_{1},\ldots,X_{i_{k}}=t_{k},X_{k+1:m}>t). \label{DefMCHR}%
\end{equation}
Furthermore, we put
\begin{equation}
\lambda_{j}(t|\emptyset):=\lim_{\Delta t\rightarrow0^{+}}\frac{1}{\Delta
t}\mathbb{P}(X_{j}\leq t+\Delta t|X_{1:m}>t). \label{DefMCHR2}%
\end{equation}

\end{definition}

For remarks, details, and for general aspects concerning this definition see
e.g. \cite{ShaSha90}, \cite{ShaSha}, \cite{Spi19ASMBI}, the review paper
\cite{ShaSha15}, and references cited therein. It has been pointed out in
\cite{DMS20} that the set of the multivariate conditional hazard rate
functions is a convenient tool to describe some aspects related with the
quantities $\alpha_{j}(A)$ in (\ref{alfas}). It will turn out that such a
description is specially convenient for our purposes, as well. This choice, in
particular, leads us to single out the class of (time-homogeneous)
Load-Sharing dependence models and to appreciate related role in the present context.

For lifetimes $X_{1},\ldots,X_{m}$, load-sharing is a simple condition of
stochastic dependence which is defined in terms of the m.c.h.r. functions and
which has a long history in Reliability Theory. See e.g. \cite{Spi19ASMBI} for
references and for some more detailed discussion and demonstration. Such a
condition amounts to imposing that the m.c.h.r. functions $\lambda_{j}
(t|i_{1},\ldots,i_{k};t_{1},\ldots,t_{k})$ do not depend on the arguments
$t_{1},\ldots,t_{k}$. Here, we concentrate attention on the following specific definition.

\begin{definition}
The $m$-tuple $\left(  X_{1},\ldots,X_{m}\right)  $ is distributed according
to a \textit{load-sharing model} when, for $k\in\lbrack m-1]$, $(i_{1}%
,\ldots,i_{k},j)\in\mathcal{D}(\emptyset,k+1)$, and for an ordered sequence
$0<t_{1}<\cdots<t_{k}<t,$ one has
\begin{equation}
\lambda_{j}(t|i_{1},\ldots,i_{k};t_{1},\ldots,t_{k})=\mu_{j}\{i_{1}%
,\ldots,i_{k}\},\lambda_{j}(t|\emptyset)=\mu_{j}(\emptyset), \label{DefTHLSM}%
\end{equation}
for suitable functions $\mu_{j}\{i_{1},\ldots,i_{k}\}$ and quantities $\mu
_{j}(\emptyset)$.
\end{definition}

In (\ref{DefTHLSM}) it is intended that, for fixed $j$ $\in\lbrack m]$, the
function $\mu_{j}\{i_{1},\ldots,i_{k}\}$ does not depend on the order
according to which $i_{1},\ldots,i_{k}$ are listed. One can however admit the
possibility for the function $\mu_{j}$ to depend on the ordering
of\ $i_{1},\ldots,i_{k}$. To distinguish the latter case from the one in
(\ref{DefTHLSM}), we shall use the notation $\mu_{j}(i_{1},\ldots,i_{k})$ in
place of $\mu_{j}\{i_{1},\ldots,i_{k}\}$ and give the following

\begin{definition}
The $m$-tuple $\left(  X_{1},\ldots,X_{m}\right)  $ is distributed according
to an \textit{order-dependent} \textit{load-sharing model} when, for
$k\in\lbrack m-1]$, $(i_{1},\ldots,i_{k},j)\in\mathcal{D}(\emptyset,k+1)$, and
for an ordered sequence $0<t_{1}<\cdots<t_{k}<t,$ one has
\begin{equation}
\lambda_{j}(t|i_{1},\ldots,i_{k};t_{1},\ldots,t_{k})=\mu_{j}(i_{1}%
,\ldots,i_{k}),\lambda_{j}(t|\emptyset)=\mu_{j}(\emptyset
),\label{SymbForTHLSM}%
\end{equation}
\ for suitable functions $\mu_{j}(i_{1},\ldots,i_{k})$ and quantities $\mu
_{j}(\emptyset)$.
\end{definition}

\textit{ }

A slightly different formulation of the above concept has been given in the
recent paper \cite{FosEtAl21}. Even though not very natural in the engineering
context of systems' reliability, the possibility of considering
order-dependence is potentially interesting both from a mathematical viewpoint
and for different types of applications. In particular order-dependent
load-sharing models will be met in the next Theorem \ref{lemma2} and have
emerged in \cite{FosEtAl21}, when dealing with the construction of probability
models that satisfy some symmetry properties weaker than exchangeability, and
that are anyway non-exchangeable.

When the order-dependent case is excluded, and for $I=\{i_{1},\ldots
,i_{k}\}\subset\left[  m\right]  $, it will be often convenient also to use
the symbols $\mu_{j}(I)$ with the following meaning:%
\[
\mu_{j}(I):=\mu_{j}\{i_{1},\ldots,i_{k}\}.
\]

A further way (different from the one of order-dependence) of weakening the
condition (\ref{DefTHLSM}), is the one of \textquotedblleft non-homogeneous
load-sharing". This term refers to cases where the functions $\lambda
_{j}(t|\mathbf{\cdot};\mathbf{\cdot})$, and $\lambda_{j}(t|\emptyset)$,
respectively defined in (\ref{DefMCHR}), still do not depend on $t_{1}%
,\ldots,t_{k}$, but depend however on the argument $t$. In this paper we do
not need this type of generalization.

For a fixed family $\mathcal{M}$ of parameters $\mu_{j}$'s, for $k\in\lbrack
m-1]$, and for $(i_{1},\ldots,i_{k} )\in\mathcal{D}(\emptyset,k)$, set
\begin{equation}
M(i_{1},\ldots,i_{k}):=\sum_{j\in\lbrack m]\setminus\{i_{1},\ldots,i_{k}\}}%
\mu_{j}(i_{1},\ldots,i_{k})\text{ and }M(\emptyset)=\sum_{j\in\lbrack m]}%
\mu_{j}(\emptyset). \label{muode}%
\end{equation}

As a relevant property of (possibly order-dependent) load sharing models, one
has $\mathbb{P}(J_{1}=j)=\frac{\mu_{j}(\emptyset)}{\text{ }M(\emptyset)}$ and
the above formula (\ref{ProbCondizPerJ}) reduces to the simple identity:%
\begin{equation}
\mathbb{P}(J_{k+1}=j|J_{1}=i_{1},\,J_{2}=i_{2},\,\ldots,\,J_{k}=i_{k}%
)=\frac{\mu_{j}(i_{1},\ldots,i_{k})}{M(i_{1},\ldots,i_{k})}
\label{ProbConditJinLS}%
\end{equation}
(see also \cite{Spi19ASMBI} and \cite{DMS20}). A very simple form then follows
for the probability
\[
p_{k}^{(m)}(i_{1},\ldots,i_{k})=\mathbb{P}(J_{1}=i_{1},\,J_{2}=i_{2}%
,\,\ldots,\,J_{k}=i_{k}),
\]
\ for which we can immediately obtain

\begin{lemma}
\label{lemma1} Let $\left(  X_{1},\ldots,X_{m}\right)  $ follow an
(order-dependent) load-sharing model with the family of coefficients
$\mathcal{M}$. Let $k\in\lbrack m]$ and let $(i_{1},\ldots,i_{k}%
)\in\mathcal{D}(\emptyset,k)$. Then
\begin{equation}
p_{k}^{(m)}(i_{1},\ldots,i_{k})=\frac{\mu_{i_{1}}(\emptyset)}{M(\emptyset
)}\frac{\mu_{i_{2}}(i_{1})}{M(i_{1})}\frac{\mu_{i_{3}}(i_{1},i_{2})}%
{M(i_{1},i_{2})}\ldots\frac{\mu_{i_{k}}(i_{1},i_{2},\ldots i_{k-1})}%
{M(i_{1},i_{2},\ldots i_{k-1})}. \label{dispo}%
\end{equation}

\end{lemma}

\medskip

Notice that, for $k = m$, $p_{m}(i_{1},\ldots, i_{m-1},i_{m}) = p_{m}%
(i_{1},\ldots, i_{m-1}) $ thus $p_{m}(i_{1},\ldots, ,i_{m}) $ is not
influenced by $\mu_{i_{m}} (i_{1}, \ldots, i_{m-1})$.

The previous result had already been stated as Proposition 2 in
\cite{Spi19ASMBI}), relating to the special case when the order-dependence
condition is excluded.

As a consequence of Proposition \ref{prop1} and the above Lemma, we can state
the following

\begin{prop}
\label{Proposition Above} Let $\left(  X_{1},\ldots,X_{m}\right)  $ follow an
(order-dependent) load-sharing model with the family of coefficients
$\mathcal{M}$. Let $A\in\widehat{\mathcal{P}}(m)$ with $\ell=|A|$.
Then, for $j\in A$ one has
\[
\alpha_{j}(A)=\mathbb{P}(X_{j}=\min_{i\in A}X_{i})=\frac{\mu_{j}(\emptyset
)}{M(\emptyset)}+
\]%
\begin{equation}
+\sum_{k=1}^{m-\ell}{\sum_{(i_{1},\ldots,i_{k})\in\mathcal{D}( A,k)}}\frac
{\mu_{i_{1}}(\emptyset)}{M(\emptyset)}\frac{\mu_{i_{2}}(i_{1} )}{M(i_{1}%
)}\ldots\frac{\mu_{i_{k}}(i_{1},i_{2},\ldots i_{k-1})}{M(i_{1} ,i_{2},\ldots
i_{k-1})}\frac{\mu_{j}(i_{1},i_{2},\ldots i_{k})}{M(i_{1} ,i_{2},\ldots
i_{k})}. \label{formulata}%
\end{equation}

\end{prop}

A direct implication of the above Proposition is the following

\begin{corollary}
\label{Dipende} Let $\left(  X_{1},\ldots,X_{m}\right)  $ follow an
(order-dependent) load-sharing model with family of coefficients $\mathcal{M}%
$. For given $A\in\widehat{\mathcal{P}}(m)$, the probabilities $\{\alpha
_{j}(A):j\in A\}$ only depend on $\{\mu_{h}(I):I\subset A^{c},h\not \in I\}$.
\end{corollary}

\bigskip

Consider now an arbitrary probability distribution $\rho^{\left(  m\right)  }$
on the set of permutations $\Pi_{m}\equiv\mathcal{D} (\emptyset, m) $. The
forthcoming result shows the existence of some order-dependent load-sharing
model such that the corresponding joint distribution of the vector $\left(
J_{1},\ldots,J_{m}\right)  $ coincides with $\rho^{\left(  m\right)  }$.

\begin{theorem}
\label{lemma2} For $m\geq2$ let the function $\rho^{\left(  m\right)  }%
:\Pi_{m}\rightarrow\lbrack0,1]$ satisfy the condition
\[
\sum_{(j_{1},\ldots,j_{m})\in\Pi_{m}}\rho^{\left(  m\right)  }(j_{1}%
,\ldots,j_{m})=1.
\]
Then there exist an (order-dependent) load-sharing model with family of
coefficients $\mathcal{M} $
%
such that
\[
p_{m}(j_{1},\ldots,j_{m})=\rho^{\left(  m\right)  }(j_{1},\ldots,j_{m}).
\]

\end{theorem}

\begin{proof}
For the fixed function $\rho^{\left(  m\right)  }$ and for $(j_{1}%
,\ldots,j_{k})\in\mathcal{D}(\emptyset,k)$ we set
\begin{equation}
w(j_{1},\ldots,j_{k})=\sum_{(i_{1},\ldots,i_{m-k})\in\mathcal{D}%
(\{j_{1},\ldots,j_{k}\},m-k)}\rho^{\left(  m\right)  }(j_{1},\ldots
,j_{k},i_{1},\ldots,i_{m-k}). \label{W2}%
\end{equation}
As suggested by the above formula (\ref{ProbConditJinLS}) we fix now the
family $\mathcal{M}$ formed by the parameters given as follows
\[
\mu_{j}(\emptyset)=w(j),\text{ }\mu_{j_{2}}(j_{1})=\frac{w(j_{1},j_{2}%
)}{w(j_{1})},\text{ }\mu_{j_{3}}(j_{1},j_{2})=\frac{w(j_{1},j_{2},i_{3}%
)}{w(j_{1},j_{2})},\ldots,
\]%
\begin{equation}
\ldots,\mu_{j_{m-1}}(j_{1},j_{2},\ldots,j_{m-2})=\frac{w(j_{1},\ldots
,j_{m-1})}{w(j_{1},\ldots,j_{m-2})}. \label{ConstructionOfMi}%
\end{equation}
In the previous formula we tacitly understand $0/0=0$. For the order-dependent
load-sharing model corresponding to $\mathcal{M}$ above, the proof can be
concluded by just applying Lemma \ref{lemma1}.
\end{proof}

\begin{example}
\label{doremifa} Here we continue the Example \ref{doremi}. By taking into
account the assessment of the values $p(j_{1},j_{2},j_{3})$ therein and
recalling the position \eqref{W2}, we set
\[
w(j_{1},j_{2},j_{3})=p(j_{1},j_{2},j_{3}),\forall(j_{1},j_{2},j_{3})\in\Pi
_{3},
\]%
\[
w(1,2)=\frac{2}{18},\,\,\,w(2,1)=\frac{2}{18},\,\,\,w(3,2)=\frac{5}{18},
\]%
\[
w(3,1)=\frac{3}{18},\,\,\,w(2,3)=\frac{2}{18},\,\,\,w(1,3)=\frac{4}{18},
\]%
\[
w(1)=w(1,2)+w(1,3)=\frac{1}{3},
\]%
\[
w(2)=w(2,1)+w(2,3)=\frac{2}{9},w(3)=w(3,1)+w(3,2)=\frac{4}{9}.
\]

Whence, by applying (\ref{ConstructionOfMi}), we obtain%
\[
\mu_{1}(2)=\frac{w(2,1)}{w(2)}=\frac{1}{2};\mu_{1}(3)=\frac{w(3,1)}
{w(3)}=\frac{3}{8};
\]
\[
\mu_{2}(1)=\frac{w(1,2)}{w(1)}=\frac{1}{3};\mu_{2}(3)=\frac{w(3,2)}
{w(3)}=\frac{5}{8};
\]
\[
\mu_{3}(1)=\frac{w(1,3)}{w(1)}=\frac{2}{3};\mu_{3}(2)=\frac{w(2,3)}
{w(2)}=\frac{1}{2}.
\]
Finally, we can set
\[
\mu_{j_{1}}(j_{2},j_{3})=1,\forall(j_{1},j_{2},j_{3})\in\Pi_{3}.
\]

\end{example}

Let us now concentrate on the non-order-dependent case. For fixed
$I=\{i_{1},\ldots,i_{k}\}$, consider now the set formed by the $\left(
m-k\right)  $\ values $\mu_{j}(i_{1},\ldots,i_{k})$, for $j\notin%
\{i_{1},\ldots,i_{k}\}$.

Generally such a set of values actually depends on $I$. But there are
interesting cases where, for any subset $I\subset\lbrack m]$, the collection
of coefficients $\left\{  \mu_{j}(I):j\notin I\right\}  $ depends on $I$ only
through its cardinality $|I|$, namely:
\begin{equation}
\left\{  \mu_{j}(I):j\notin I\right\}  \equiv\left\{  \mu_{j}%
(\{1,2,...,|I|\}:j\neq1,...,|I|)\right\}  . \label{Ex Set-Invariance Property}%
\end{equation}
In such cases constants $\widehat{M}_{1},...,\widehat{M}_{m}$ exist so that%
\begin{equation}
M\{i_{1},...,i_{k}\}=M\{1,...,k\}=\widehat{M}_{k}. \label{Mcostante}%
\end{equation}
Furthermore, we set $\widehat{M}_{0} = M (\emptyset)$.

The family $\mathcal{M}$ constructed along the proof of Theorem 1 does
generally correspond to an order-dependent load-sharing model, and this
excludes the possibility of the condition in
\eqref{Ex Set-Invariance Property}. Even if very special, on the other hand,
the class of models satisfying (\ref{Ex Set-Invariance Property}) will have a
fundamental role in the next two sections.

\section{Existence and construction of Load Sharing models concordant with
ranking patterns}

\label{sec4} Let $\boldsymbol{\sigma}\in\hat{\Sigma}^{(m)}$ be an assigned
ranking pattern. In this section we aim to construct, for lifetimes
$X_{1},...,X_{m}$, a probabilistic model $p$-concordant with
$\boldsymbol{\sigma}$, according to the definition given in Section 2. In
other words, by looking at the probabilities $\alpha_{j}(A)$, we search joint
distributions for $X_{1},...,X_{m}$ such that the equivalence in
(\ref{glialfa}) holds. The existence of such distributions will be in fact
proven here. Actually we will constructively identify some of such
distributions, and this task will be accomplished by means of a search within
the class of load-sharing models with special parameters satisfying the
condition (\ref{Ex Set-Invariance Property}).

More specifically, we introduce a restricted class of load-sharing models by
starting from the assigned ranking pattern $\boldsymbol{\sigma}$. Such a class
fits with our purposes and is defined as follows.

\begin{definition}
\label{defLSEpsilonModels} Let $\varepsilon(2),...,\varepsilon(m)$ be positive
quantities such that
\[
(\sigma(A,i)-1)\varepsilon(|A|)<1
\]
for all $A\in\cP(m)$ and $i\in A$. A $LS\left(\boldsymbol{\varepsilon},\boldsymbol{\sigma}\right)  $  is defined by
parameters of the form
\begin{equation}
{\mu}_{i}([m]\setminus A)=1-(\sigma(A,i)-1)\varepsilon(|A|), \text{ }
A\in\cP(m),\text{ }i\in A \label{scelta1}%
\end{equation}
For $A=\{i\}$ we finally set $\mu_{i}([m]\setminus A)=1$, so that
$\varepsilon(1)=0$.
\end{definition}

As it is possible to prove, in fact, $\varepsilon(2),...,\varepsilon(m)$ can
be adequately fixed in order to let the model $LS\left(
\boldsymbol{\varepsilon},\boldsymbol{\sigma}\right)  $ to satisfy the
condition (\ref{glialfa}). In this direction, we first point out the following
features of such models.

We notice that the numbers $\mu_{j}([m]\setminus A)$ (for $j\in A$) are the
same for all the subsets with same cardinality $h=|A|$. Thus the identities in
(\ref{Mcostante}) hold for $LS\left(  \boldsymbol{\varepsilon}%
,\boldsymbol{\sigma}\right)  $, in view of the validity of
(\ref{Ex Set-Invariance Property}). More precisely, by \eqref{muode}, one can
write
\begin{equation}
\widehat{M}_{m-h}=\sum_{u=1}^{h}[1-(u-1)\varepsilon(m)]=h-\frac{h(h-1)}%
{2}\varepsilon(h), \label{Mdef}%
\end{equation}
for $h \in[m]$.

As an application of Corollary \ref{Dipende}, we can realize that, for given
$B\in\widehat{\mathcal{P}}( m )$ with $|B|=n$ and $j\in B$, the probability
$\alpha_{j}(B)$ only depends on $\varepsilon(n),\varepsilon
(n+1)...,\varepsilon(m)$ and on the functions $\sigma(D, \cdot)$ for
$D\in\widehat{P}(m) $ with $D\supset B$.

On this basis it is possible to prove the following existence result.

\begin{theorem}
\label{th1semplificato} For $m\in\mathbb{N}$, let $\boldsymbol{\sigma}\in
\hat{\Sigma}^{(m)}$ be a ranking pattern. Then there exist constants
$\varepsilon(2),...,\varepsilon(m)$ such that $\boldsymbol{\sigma}$ is
$p$-concordant with a $m$-tuple $(X_{1},\ldots,X_{m})$ distributed according
to the model $LS\left(  \boldsymbol{\varepsilon},\boldsymbol{\sigma}\right)
$, where $\boldsymbol{\varepsilon}=\left(  \varepsilon(1),...,\varepsilon
(m)\right)  $.
\end{theorem}

We notice that Theorem \ref{lemma2} shows the general interest of load-sharing
models in the present context, but it leaves unsolved the problem whether, for
an arbitrary ranking pattern $\boldsymbol{\sigma}\in\hat{\Sigma}^{(m)}$, it is
possible to find a distribution $\rho$ able to generate $\boldsymbol{\sigma}$.
Such a problem is solved by Theorem \ref{th1semplificato} and the solution can
be obtained in terms of models of the type $LS(\boldsymbol{\varepsilon
},\boldsymbol{\sigma})$.

Here, we are not going to give the proof of such existence result. At the cost of a more technical procedure, we
are rather going to prove the following quantitative result which simultaneously shows the existence of the wanted $LS(\boldsymbol{\varepsilon
},\boldsymbol{\sigma})$ models and suggests how to construct appropriate choices for $\boldsymbol{\varepsilon}$.

\begin{theorem}
\label{EpsilonEspliciti} For  for any $\boldsymbol{\sigma}\in\hat{\Sigma}$ and
any $\boldsymbol{\varepsilon}=(\varepsilon(1),\ldots,\varepsilon(m))$ such
that, for $\ell=1,...,m-1,$
\begin{equation}
\frac{(m-\ell)!(\ell-1)!}{2\cdot m!}\varepsilon(\ell)>8\ell\varepsilon
(\ell+1)\label{stimaepsilon}%
\end{equation}
the model LS$(\boldsymbol{\varepsilon},\boldsymbol{\sigma})$ is $p$-concordant
with $\boldsymbol{\sigma}$.
%

\end{theorem}

\bigskip

The inequalities in \eqref{stimaepsilon} can be e.g. obtained by simply letting, for $l=1,...,m$,
\begin{equation}
\varepsilon(\ell)=(17\cdot m\cdot m!)^{-\ell+1}.\label{stimaepsilon2}%
\end{equation}
We notice that the form of the coefficients $\varepsilon(1),\ldots
,\varepsilon(m)$ is universal, in the sense that it is independent of the
ranking pattern $\sigma$ and can then be fixed a priori. Obviously the
generated intensities $\mu$'s, characterizing the $p$-concordant load-sharing
model, depend on both $\boldsymbol{\varepsilon}$ and $\boldsymbol{\sigma}$.

As a result of the arguments above, one can now conclude with the following conclusion.

Let $\boldsymbol{\sigma}\in\hat{\Sigma}^{(m)}$ be a given ranking pattern,
then it is $p$-concordant with an $m$-tuple $(X_{1},\ldots,X_{m})$ distributed
according to a load sharing model with coefficients of the form
\eqref{scelta1} and more precisely given by
\begin{equation}
\mu_{j}([m]\setminus A)=1-\frac{\sigma(A,j)-1}{(17\cdot m\cdot m!)^{|A|-1}%
},j\in A,\label{qspiega}%
\end{equation}

where $A$ is a non-empty subset of $[m]$.
\begin{example}
The above conclusion can, for instance, be applied to the search of load-sharing models that are
$p$-concordant with the paradoxical ranking patterns $\boldsymbol{\sigma}$
\ which have been presented in the Example \ref{VeryParadox}. Let us consider,
for any such $\boldsymbol{\sigma}$, the related model
$LS(\boldsymbol{\varepsilon,\sigma})$ with the vector $\boldsymbol{\varepsilon
}$ \ of the special form given in \eqref{stimaepsilon2}. By taking into
account \eqref{qspiega} we can obtain that all such model are characterized by
the following common conditions: for $A=\{1,j\}$ with $j\neq1$
\[
{\mu}_{1}([m]\setminus A)=1,
\]%
\[
{\mu}_{j}([m]\setminus A)=1-(17\cdot m\cdot m!)^{-1},
\]
while, for $A$ such that $1\in A$ and $|A|=\ell>2$,
\[
{\mu}_{1}([m]\setminus A)=1-(\ell-1)(17\cdot m\cdot m!)^{-\ell+1}.
\]
The other intensities will on the contrary depend on the choice of any special
$\boldsymbol{\sigma}$. \bigskip
\end{example}

The proof of Theorem \ref{EpsilonEspliciti} will be given below and will be
based upon some technical properties, concerning $LS\left(
\boldsymbol{\varepsilon},\boldsymbol{\sigma}\right)  $ models, which we are
going to show henceforth.

Preliminarily it is convenient to recall that, for a generic Load Sharing
model, the quantities $\alpha_{j}(A)$ take the form \eqref{formulata} and that
the special structure of $LS(\boldsymbol{\varepsilon},\boldsymbol{\sigma})$
allows us to reduce the construction of the wanted models to the
identification of a suitable vector $\boldsymbol{\varepsilon}$.

A path to achieve such a goal is based on a suitable decomposition of
$\alpha_{j}(A)$ into two terms (see \eqref{abc}) and on showing that one of
such two terms can be made dominant with respect to the other.

First, it is useful to require that $\boldsymbol{\varepsilon}$ satisfy the
conditions
\[
\varepsilon(2)<\frac{1}{4}
\]
and%
\begin{equation}
\text{ }2(u-1)\varepsilon(u)<(u-2)\varepsilon(u-1),\text{ for }u=2,\ldots,m.
\label{condizioneH1}%
\end{equation}
We notice that the latter condition is implied by \eqref{stimaepsilon}.

Furthermore, we also introduce the following alternative symbols which will be
sometimes used, when more convenient, in place of the $\varepsilon$'s: for
$u=2,\ldots,m$,
\begin{equation}
\rho(u)=\varepsilon(u)\frac{(u-1)}{2}. \label{rho}%
\end{equation}

Written in terms of the $\rho$'s, the condition in \eqref{condizioneH1}
becomes
\begin{equation}
\rho(2)<\frac{1}{8}\text{ and }2\rho(u)<\rho(u-1),\text{ for }u=2,\ldots
,m\label{coH1}
\end{equation}
and the following simple consequence of \eqref{coH1} will be used several
times along the forthcoming proofs:
\begin{equation}
\sum_{u=k}^{m}\rho(u)<2\rho(k),\label{StimaSomma}
\end{equation}
for $k=2,\ldots,m$. We are thus ready to present useful inequalities, in the lemma below.

\begin{lemma}\label{alto-basso}
Let $m \geq 2$, $\boldsymbol{\sigma}\in\hat{\Sigma}^{(m)}$,
 and let the $m$-tuple $\left(  X_{1},\ldots , X_{m}\right) $
 be distributed according to a model
$LS\left(  \boldsymbol{\varepsilon},\boldsymbol{\sigma}\right) $.
Then, for any $k \in [m]$,
\begin{equation}
\label{finalebis}
 \mathbb{P}  ( J_1 = i_1, \,  J_2 =i_2, \, \ldots  , \,  J_k = i_k    )
  \leq  \frac{(m-k )!}{m!}
 \left  (    1+ 2 \sum_{u =m-k+1}^m \rho (u )     \right   )
\end{equation}
and
\begin{equation}
\label{dalbasso3}
 \mathbb{P}  ( J_1 = i_1, \,  J_2 =i_2, \, \ldots  , \,  J_k = i_k    )
  \geq  \frac{(m-k )!}{m!}
 \left  (    1- 2 \sum_{u =m-k+1}^m \rho (u )     \right   )  .
\end{equation}
\end{lemma}
\begin{proof}
We start by proving the inequality \eqref{finalebis}. For $k\in\lbrack m]$, by
taking into account formula \eqref{dispo}, \eqref{scelta1},
we obtain the following equality
\[
\mathbb{P}(J_{1}=i_{1},\,J_{2}=i_{2},\,\ldots,\,J_{k}=i_{k})=
\]%
\[
=\frac{\mu_{i_{1}}(\emptyset)}{m-\frac{m(m-1)}{2}\varepsilon(m)}\times
\frac{{\mu}_{i_{2}}(i_{1})}{m-1-\frac{(m-1)(m-2)}{2}\varepsilon(m-1)}%
\times\ldots
\]%

\medskip

\[
\ldots\times\frac{\mu_{i_{k}}(i_{1},i_{2},\ldots i_{k-1})}{m-k+1-\frac
{(m-k+1)(m-k)}{2}\varepsilon(m-k+1)}.%
\]
Since all the coefficients ${\mu}_{i}$ with $i\in I$ are smaller than 1, we
obtain
\[
\mathbb{P}(J_{1}=i_{1},\,J_{2}=i_{2},\,\ldots,\,J_{k}=i_{k})\leq
\]%

\[
\leq\frac{1}{m-\frac{m(m-1)}{2}\varepsilon(m)}\times\frac{1}{m-1-\frac
{(m-1)(m-2)}{2}\varepsilon(m-1)}\times\ldots
\]%
\[
\ldots\times\frac{1}{{m-k+1-\frac{(m-k+1)(m-k)}{2}\varepsilon(m-k+1)}}=
\]

\[
=\frac{1}{m(m-1)\cdots(m-k+1)}\times\frac{1}{1-\rho(m)}\times\frac{1}%
{1-\rho(m-1)}\times\ldots
\]
\begin{equation}
\ldots\times\frac{1}{{1-\rho(m-k+1)}}.\label{stima3}%
\end{equation}
By \eqref{StimaSomma} one has $\sum_{k=2}^{m} \rho (k) < \frac{1}{4}   <\frac{1}{2}$. Furthermore, for $a \in (0, \frac{1}{2} ) $, the inequality $1/ (1-a) < 1+2a$ holds.

Hence, we can conclude that the quantity in \eqref{stima3} is less or equal than
\begin{equation*}\label{finale}
  \frac{1}{m (m-1)\cdots (m-k+1)}\times  \frac{1}{1 - \sum_{ u =m-k+1}^m \rho (u )} \leq
  \frac{     1+ 2 \sum_{u =m-k+1}^m \rho (  u )         }{m (m-1)\cdots (m-k+1)}=
\end{equation*}
\begin{equation*}
 =  \frac{(m-k )!}{m!}
\left  (    1+ 2 \sum_{u =m-k+1}^m \rho (u )     \right   )  .
\end{equation*}

\medskip

We now prove the inequality \eqref{dalbasso3}.
By taking again into account formulas \eqref{dispo}, \eqref{scelta1} and \eqref{rho}, we
also  can give the following lower bound
\begin{equation*}\label{dalbasso}
\mathbb{P}   ( J_1 = i_1, \,  J_2 =i_2, \, \ldots  , \,  J_k = i_k    )   \geq  \frac{   \mu_{i_1}    (\emptyset ) \times   \mu_{i_2}
 (i_1)  \times \cdots \times  \mu_{i_{k}}
  (i_1, i_2 , \ldots i_{m-1}         )     }{m(m-1) \cdots (m-k+1)} \geq
\end{equation*}
\medskip
\begin{equation*}\label{dalbasso2}
\geq   \frac{   [1- (m-1) \varepsilon (m)] \times  [1- (m-2) \varepsilon (m-1)] \times \cdots \times
  [ 1-(m-k) \varepsilon (m-k+1) ]    }{m(m-1) \cdots (m-k+1)} \geq
\end{equation*}
\medskip
\begin{equation*}
\geq   \frac{   1- \sum_{u =m-k+1}^m (u-1) \varepsilon (u )    }{m(m-1) \cdots (m-k+1)}
 =   \frac{(m-k)!}{m!}
\left  (    1- 2 \sum_{u =m-k+1}^m \rho (u )     \right   )     .
\end{equation*}
\end{proof}

\medskip

For an $m$-tuple $\left(  X_{1},\ldots , X_{m}\right)$ distributed according to a Load Sharing model $LS\left( \boldsymbol{\varepsilon},\boldsymbol{\sigma}\right) $, the probabilities in \eqref{dispo} depend on the pair
$\boldsymbol{\varepsilon}, \boldsymbol{\sigma}$.
 Then also the probabilities $\alpha_i (A)$ in \eqref{formula} are determined by $\boldsymbol{\varepsilon}, \boldsymbol{\sigma}$.

 Related to the $m$-tuple $\left(  X_{1},\ldots , X_{m}\right)$ and to the corresponding vector $(J_1 , \ldots J_m)$,
we now aim to give the probabilities $\alpha_j (A)$  an expression convenient for what follows. We shall
use the symbol $ \alpha_j (A, \boldsymbol{\sigma})$ and, in order to apply Proposition \ref{prop1}, we also  introduce the following notation. Fix $A \in \cP (m)$, for $i \in A$ and $ \ell = |A|\leq m-1 $,
 $$
 \beta_i (A , \boldsymbol{\sigma}) :  = \mathbb{P}  (J_1 =i )  , \text{ if } \ell = m -1 ,
$$
\begin{equation}\label{b2}
  \beta_i (A   , \boldsymbol{\sigma}) := \mathbb{P}  (J_1 =i ) + \sum_{k= 1}^{m- \ell-1}  {\sum_{     (i_1, \ldots , i_{k} )\in \mathcal{D} (A, k) }}
   \mathbb{P}  (J_1 =i_1,  \, J_2 =i_2, \, \ldots , \, J_{k} = i_{k} , \, J_{k+1} = i )      ,
\end{equation}
if $2 \leq   \ell  \leq m -2 $. Also we denote by $ \gamma_i (A   , \boldsymbol{\sigma}   ) $ the probability of the intersection
\[
\{X_{i}=\min_{j\in A}X_{j}\}\cap\{X_{i}>\max_{j\in A^{c}}X_{j}\},
\]
i.e.
\begin{equation}\label{c2}
  \gamma_i (A   , \sigma    ) :=   {\sum_{     (i_1, \ldots , i_{m - \ell} )\in \mathcal{D} ( A, m- \ell) }} \mathbb{P}
   (J_1 =i_1,  \, J_2 =i_2, \, \ldots , \, J_{n - \ell} = i_{m - \ell } , \, J_{m - \ell +1} = i )  ,
\end{equation}
for any $2 \leq   \ell  \leq m -1 $.

In terms of this notation and  recalling Proposition \ref{prop1}, we can now write
\begin{equation}\label{abc}
 \alpha_i (A, \boldsymbol{\sigma}) = \beta_i (A ,\boldsymbol{\sigma} ) + \gamma_i (A, \boldsymbol{\sigma}) ,
\end{equation}
for $A \in \cP (m)$ with $|A| \leq m-1$.


By recalling Corollary \ref{Dipende} and the position \eqref{scelta1},  we  see that $  \beta_i (A   , \boldsymbol{\sigma} ) $ only depends on
\begin{equation}\label{depende1}
\{     \sigma (B, \cdot  ) : |B | \geq \ell  \} ,
\end{equation}
and
$    \gamma_i (A   , \sigma  ) $ only depends on
\begin{equation}\label{depende2}
\{     \sigma   (B, \cdot  ) : |B | \geq \ell -1 \} .
\end{equation}

Consider now, for fixed $A \in \cP (m)$ such that $|A| = \ell $,
\begin{equation}\label{maxbeta}
\mathcal{B} (\ell ) :=\max \{   \beta_i (A   , \boldsymbol{\sigma} ) -   \beta_j (A   , \boldsymbol{\sigma}  ) \} ,
\end{equation}
where the maximum is computed with respect to all the ranking patterns  $\boldsymbol{\sigma}\in\hat{\Sigma}^{(m)}$.
 It follows  that $ \mathcal{B} (\ell )$ only depends on the quantities in \eqref{depende1} and
on    the cardinality $\ell = |A| $. It does not  depend on $i,j, A$.


\medskip

On the other hand, we also consider the quantities
\begin{equation}\label{mingamma}
 \mathcal{C} (\ell ) : =\min \{ \gamma_i (A   , \boldsymbol{\sigma}  )  - \gamma_j (A   , \boldsymbol{\sigma} )
  \}   >0  ,
\end{equation}
where the minimum is computed over the family of all the ranking patterns $\boldsymbol{\sigma}\in\hat{\Sigma}^{(m)}$ such that $\sigma ( A, i) <   \sigma (A, j) $.
 It follows  that $ \mathcal{C} (\ell )$ only depends on the quantities in \eqref{depende2}.
Similarly to above,  $  \mathcal{C} (\ell )$ does not depend on $i,j,A$. However it
 depends on the cardinality $\ell = |A| $.

Obviously, all the quantities $ \alpha_j (A , \boldsymbol{\sigma})$,
$ \beta_j (A,\boldsymbol{\sigma})$, $\gamma_j (A , \boldsymbol{\sigma})$, $ \mathcal{B} (\ell )$ and
$ \mathcal{C} (\ell)$ depend on the vector $\boldsymbol{\varepsilon}$. At this point, referring to \eqref{abc},  we aim to show
that $\boldsymbol{\varepsilon}$ can be suitably chosen in such a way that  $\gamma_j $ gives a relevant contribution  in
imposing a comparison between the two values $\alpha_i (A, \boldsymbol{\sigma})$ and $\alpha_j (A, \boldsymbol{\sigma})$.
On this purpose, we can hinge on the following result.

\begin{lemma}\label{lemmaBC}
For any $m \geq 3$ and any $\boldsymbol{\sigma}\in\hat{\Sigma}^{(m)}$
 one has, for $\ell = 2, \ldots , m-1$, that
\begin{equation}\label{Bfine}
 \mathcal{B}(\ell) \leq 8 \ell\varepsilon (\ell +1)
\end{equation}
and
\begin{equation}\label{Cbasso5}
 \mathcal{C}    (\ell ) \geq
   \frac{   (m-\ell) !        (\ell -1)!}{   2 \cdot m!}   \varepsilon (\ell ) .
\end{equation}
\end{lemma}
\begin{proof}
For  $A \in \cP (m)$ such that  $|A| = \ell < m$ and $i,j \in A$ one has, by definition of $\mathcal{B} (\ell)$ and by \eqref{b2}, that
\begin{equation*}\label{Bell}
\mathcal{B} (\ell ) =\max_{ \boldsymbol{\sigma} \in \hat{\Sigma}^{(m)} } \{ \beta_i (A , \boldsymbol{\sigma}) -   \beta_j (A   , \boldsymbol{\sigma}) \} \leq
\end{equation*}
$$
 \leq \max_{ \boldsymbol{\sigma} \in \hat{\Sigma}^{(m)}  } |p_1(i)- p_1(j)  | +
 $$
 \begin{equation}\label{Bell2}
 + \sum_{k= 1}^{m- \ell-1}  {\sum_{     (i_1, \ldots , i_{k} )\in \mathcal{D} (A, k) }}
 \max_{ \boldsymbol{\sigma}  \in \hat{\Sigma}^{(m)}  } | p_{k+1} (i_1, \ldots , i_k, i) - p_{k+1} (i_1, \ldots , i_k, j)  |.
\end{equation}
By \eqref{finalebis} and \eqref{dalbasso3} in Lemma \ref{alto-basso}, one has
$$
\max_{ \boldsymbol{\sigma} \in \hat{\Sigma}^{(m)} } |p_1(i)- p_1(j)  |  \leq \frac{4 \rho (m)  }{m}
$$
and
$$
\max_{ \sigma  \in \hat{\Sigma}^{(m)}  } | p_{k+1} (i_1, \ldots , i_k, i) - p_{k+1} (i_1, \ldots , i_k, j)  |
 \leq  4 \frac{(m-k-1)!  }{m !}  \sum_{u = m-k}^{m} \rho ( u)   .
$$

We can thus conclude by writing
\begin{equation} \label{BuuB}
\mathcal{B}(\ell) \leq \frac{4 \rho ( m) }{m}   + \sum_{k =1}^{m - \ell -1 } | \mathcal{D} (A, k)  | \left [4\frac{(m-k-1)!  }{m !}
 \sum_{u = m-k}^{m} \rho ( u)  \right ] .
\end{equation}
For $|A| =\ell$, notice   that $ | \mathcal{D} (A, k)  | = \frac{(m-\ell) ! }{(m- \ell -k ) !} <  \frac{   m ! }{(m -k ) !}$,
one obtains that the r.h.s. in  inequality
\eqref{BuuB} is smaller than
\begin{equation} \label{BuuB2}
    4   \rho ( m)      + 4\sum_{k =1}^{m - \ell -1 }
 \sum_{u = m-k}^{n} \rho ( u)  .
\end{equation}
%
By \eqref{StimaSomma}, the quantity in \eqref{BuuB2} is then smaller than
\begin{equation*} \label{BuuB3}
  4   \rho ( m)      + 8\sum_{k =1}^{m - \ell -1 }
 \rho ( m-k)  \leq    8\sum_{k =0}^{m - \ell -1 }
 \rho ( m-k)  \leq    16      \rho (    \ell +1 )    .
\end{equation*}
In  conclusion
\begin{equation*} 
 \mathcal{B} (\ell ) \leq    16      \rho (    \ell +1 )   = 8 \ell \varepsilon (\ell+1)   .
\end{equation*}

\smallskip

We now prove the inequality \eqref{Cbasso5}, for any $\ell =2, \ldots , m-1$. By definition of $\mathcal{C} (\ell ) $ and by \eqref{c2}, one has
\begin{equation} \label{Cbasso}
 \mathcal{C} (\ell)  \geq  {\sum_{     (i_1, \ldots , i_{m - \ell} )\in \mathcal{D} (A, m- \ell) }}
\min_{\Small{\Small{ \begin{array}{c}
       \sigma  \in \hat{\Sigma}^{(m)} :   \\
         \sigma (A, i) <   \sigma (A, j)
       \end{array}}}}
 \{ p_{m- \ell +1} (i_1, \ldots , i_{m - \ell }, i) - p_{m - \ell +1} (i_1, \ldots , i_{m- \ell }, j)  \}.
\end{equation}
We now notice, by \eqref{ProbCondizPerJ}, that the following inequality holds
$$
\min_{\Small{\Small{ \begin{array}{c}
       \boldsymbol{\sigma} \in \hat{\Sigma}^{(m)} :   \\
         \sigma     (A, i) <   \sigma  (A, j)
       \end{array}}}}
 \{ p_{m- \ell +1} (i_1, \ldots , i_{m - \ell }, i) - p_{m - \ell +1} (i_1, \ldots , i_{m- \ell }, j)  \} \geq
$$
$$
\geq
\min_{ \boldsymbol{\sigma} \in \hat{\Sigma}^{(m)} }
 \{ p_{m -\ell } (i_1, \ldots , i_{m - \ell })  \}  \times
$$
$$
 \times \min_{\Small{\Small{ \begin{array}{c}
       \boldsymbol{\sigma} \in \hat{\Sigma}^{(m)} :   \\
         \sigma (A, i) <   \sigma (A, j)
       \end{array}}}}
\{ \mathbb{P} ( J_{m -\ell +1} =i| J_1 = i_1, \ldots , J_{m -\ell } =i_{m -\ell }  ) -
\mathbb{P} ( J_{m -\ell +1} =j| J_1 = i_1, \ldots , J_{m -\ell } =i_{m -\ell }  ) \} .
$$
Whence, the r.h.s. of \eqref{Cbasso} is larger than the quantity
$$
 {\sum_{     (i_1, \ldots , i_{m - \ell} )\in \mathcal{D} (A, m- \ell) }}
\min_{ \boldsymbol{\sigma}  \in \hat{\Sigma}^{(m)} }
 \{ p_{m -\ell } (i_1, \ldots , i_{m - \ell })  \}  \times
$$
  $$
\times  \min_{\Small{\Small{ \begin{array}{c}
       \boldsymbol{\sigma} \in \hat{\Sigma}^{(m)} :   \\
         \sigma (A, i) <   \sigma  (A, j)
       \end{array}}}}
\{ \mathbb{P} ( J_{m -\ell +1} =i| J_1 = i_1, \ldots , J_{m -\ell } =i_{m -\ell }  ) -
$$
 \begin{equation}\label{Cbasso2}
 - \mathbb{P} ( J_{m -\ell +1} =j| J_1 = i_1, \ldots , J_{m -\ell } =
i_{m -\ell }  )  \}  .
\end{equation}
On the other hand, by \eqref{dalbasso3} in Lemma \ref{alto-basso}, one has
\begin{equation}\label{dalbassonuova}
\min_{ \boldsymbol{\sigma}   \in \hat{\Sigma}^{(m)} }
 \{ p_{m -\ell } (i_1, \ldots , i_{m - \ell })  \}  \geq   \left [ \frac{\ell !}{m!} (1 - 2 \sum_{u = \ell +1}^m  \rho (u)) \right ].
\end{equation}
Furthermore, by recalling Lemma \ref{lemma1} and the identity \eqref{Mdef}, we have
$$
\min_{\Small{\Small{ \begin{array}{c}
      \boldsymbol{\sigma} \in \hat{\Sigma}^{(m)} :   \\
         \sigma  (A, i) <   \sigma  (A, j)
       \end{array}}}}
\left \{ \mathbb{P} ( J_{m -\ell +1} =i| J_1 = i_1, \ldots , J_{m -\ell } =i_{m -\ell }  ) \right . -
$$
$$
\left . \mathbb{P} ( J_{m -\ell +1} =j | J_1 = i_1, \ldots , J_{m -\ell } =i_{m -\ell }  )  \right  \} =
$$
$$
= \min_{\Small{\Small{ \begin{array}{c}
      \boldsymbol{\sigma} \in \hat{\Sigma}^{(m)} :   \\
         \sigma (A, i) <   \sigma  (A, j)
       \end{array}}}}       \left \{ \frac{ \mu_i (i_1, \ldots , i_{m -\ell }) }{M (i_1, \ldots , i_{m - \ell })}  -
  \frac{ \mu_j (i_1, \ldots , i_{m -\ell })  }{M (i_1, \ldots , i_{m - \ell })}   \right  \} =
$$
$$
=\frac{ 1}{ \ell  -    \frac{\ell(\ell-1)}{2}    \varepsilon (\ell)      }  \cdot   \min_{\Small{\Small{ \begin{array}{c}
     \boldsymbol{\sigma} \in \hat{\Sigma}^{(m)} :   \\
         \sigma  (A, i) <   \sigma  (A, j)
       \end{array}}}} \{  \mu_i (i_1, \ldots , i_{m -\ell })   -    \mu_j (i_1, \ldots , i_{m -\ell })     \} \geq
$$
\begin{equation}\label{identica}
\geq
\frac{ 1}{ \ell  -    \frac{\ell(\ell-1)}{2}    \varepsilon (\ell)      }  \cdot  \varepsilon (\ell )       ,
\end{equation}
where the last inequality follows by the position \eqref{scelta1}.

In view of \eqref{Cbasso} \eqref{Cbasso2} and by combining \eqref{dalbassonuova} and \eqref{identica}, we obtain
\begin{equation*}\label{Cbasso2bis}
\mathcal{C} (\ell )  \geq   {\sum_{     (i_1, \ldots , i_{m - \ell} )\in \mathcal{D} (A, m- \ell) }} \left [ \frac{\ell !}{m!}
(1 - 2 \sum_{u = \ell +1}^m  \rho (u)) \right ]
\frac{ 1}{ \ell  -    \frac{\ell(\ell-1)}{2}    \varepsilon (\ell)      }  \    \cdot  \varepsilon (\ell ) =
\end{equation*}
\begin{equation*}
= (m- \ell ) ! \left [ \frac{\ell !}{m!}
(1 - 2 \sum_{u = \ell +1}^m  \rho (u)) \right ]
\frac{ 1}{ \ell  -    \frac{\ell(\ell-1)}{2}    \varepsilon (\ell)      }  \    \cdot  \varepsilon (\ell ) .
\end{equation*}
Then by
\[
\mathcal{C}(\ell)\geq(m-\ell)!\left[  \frac{(\ell-1)!}{m!}(1-2\sum_{u=\ell
+1}^{m}\rho(u))\right]  \varepsilon(\ell)\geq
\]%
\[
\geq\frac{(m-\ell)!(\ell-1)!}{m!}(1-4\rho(\ell+1))\varepsilon(\ell)\geq
\frac{(m-\ell)!(\ell-1)!}{2\cdot m!}\varepsilon(\ell),
\]
where we have exploited   \eqref{StimaSomma} in the second inequality and $\rho (2 ) < \frac{1}{8}$ from \eqref{coH1} in  the last step.
\end{proof}

Let us now consider an arbitrary  ranking pattern $\boldsymbol{\sigma}\in\hat{\Sigma}^{(m)}$. We are
now in a position to prove that it is possible to suitably
construct a load-sharing model for an $m$-tuple $\left(  X_{1}, \ldots  ,X_{m}%
\right)  $, such that $\boldsymbol{\sigma}$ is $p$-concordant with
$\left(X_{1},\ldots ,X_{m}\right)  $.

\begin{proof}[Proof of Theorem \ref{EpsilonEspliciti}]
For a fixed vector $\boldsymbol{\varepsilon}=\left(  \varepsilon(2),...,\varepsilon
(m)\right)  $ and a given ranking pattern $\boldsymbol{\sigma}\in\hat{\Sigma}^{(m)}$,
we here denote by $\alpha_{i}(A,\sigma ,\varepsilon)$ the probabilities
in \eqref{alfas}, corresponding to a vector $(X_{1},\ldots,X_{m})$ distributed
according to the model $LS\left(  \boldsymbol{\varepsilon},\boldsymbol{\sigma}\right)  $.
We must thus prove that, when $\varepsilon$ satisfy the condition
\eqref{stimaepsilon}, then the equivalence
\begin{equation}
\alpha_{i}(A,\boldsymbol{\sigma}  ,      \boldsymbol{\varepsilon}         )>\alpha_{j}(A,\boldsymbol{\sigma}
,       \boldsymbol{\varepsilon})\Leftrightarrow \sigma (A,i)<\sigma (A,j), \label{equivale0}
\end{equation}
holds for any $\boldsymbol{\sigma}\in\hat{\Sigma}^{(m)}$, $A\in\cP(m)$ and $i,j\in A$.

On this purpose we will first show that such a family of equivalences follows
from the validity of both the following properties:
\begin{itemize}
\item[i)] the equivalence \eqref{equivale0} imposed on the only set $A=\left[
m\right]  $, namely
\begin{equation}
\alpha_{i}([m],\boldsymbol{\sigma} ,      \boldsymbol{\varepsilon})>\alpha_{j}([m],\boldsymbol{\sigma},     \boldsymbol{\varepsilon})\Leftrightarrow\sigma ([m],i)<\sigma ([m],j),\label{equivale2}
\end{equation}
\item[ii)]  the inequalities
\begin{equation}
\mathcal{C}(\ell)=\mathcal{C}(\ell,     \boldsymbol{\varepsilon})>\mathcal{B}(\ell
)=\mathcal{B}(\ell,\boldsymbol{\varepsilon}),\text{ for $\ell=2,\ldots,m-1$,}\label{diff}
\end{equation}
where $\mathcal{C}(\ell)$ and $\mathcal{B}(\ell)$ have been introduced in
\eqref{maxbeta}, \eqref{mingamma}.
\end{itemize}
In fact, the relation \eqref{equivale2} solves the case $|A|=m$.
 When the cardinality $|A|$ belongs to $\{2,\ldots,m-1\}$, furthermore, one
has, under the condition $\sigma (A,i)<\sigma (A,j)$,
\[
0<\mathcal{C}(\ell, \boldsymbol{\varepsilon})-\mathcal{B}(\ell,\boldsymbol{\varepsilon})\leq     \alpha_{i}(A,\boldsymbol{\sigma},\boldsymbol{\varepsilon}) -
  \alpha_{j}(A, \boldsymbol{\sigma} ,\boldsymbol{\varepsilon}),
\]
which is immediately obtained by recalling equation \eqref{abc}.

It then remains to prove \eqref{equivale2} and \eqref{diff}. By Lemma
\ref{lemma1} and in view of the choice \eqref{scelta1}, one has
\[
\alpha_{i}([m],\boldsymbol{\sigma} ,\boldsymbol{\varepsilon})=\mathbb{P}(J_{1}=i)=\frac{\mu
_{i}(\emptyset)}{M(\emptyset)}=\frac{2-2(\sigma ([m],i)-1)\varepsilon
(m)}{2m-m(m-1)\varepsilon(m)},
\]
taking in particular into account the equation \eqref{Mdef} for the case $h=m$.
Then the equivalence in \eqref{equivale2} holds true provided that
$\varepsilon(m)\in(0,\frac{1}{m-1})$.

One can furthermore obtain the inequalities
$\mathcal{C}(\ell,\boldsymbol{\varepsilon})>\mathcal{B}(\ell,\boldsymbol{\varepsilon})$ for
$\ell=2,\ldots,m-1$ in view of the inequalities \eqref{Cbasso5} and
\eqref{Bfine} in Lemma \ref{lemmaBC}, provided that the condition
\eqref{stimaepsilon} holds.
\end{proof}

\section{Discussion and Concluding Remarks}

\label{secDiscussion}

Theorem \ref{lemma2} and Theorem \ref{EpsilonEspliciti} respectively show the
role of load-sharing models as possible solutions in the search of dependence
models with certain types of probabilistic features. As mentioned in the
Introduction and as we are going to sketch henceforth, such results can in
particular be applied to the study of paradoxes arising in voting theory and
of the notion of signature in the frame of systems' reliability.

Concerning with voting theory, we refer to the following standard scenario
(see e.g. \cite{Nurmi}, \cite{GL2017}, \cite{BachmeierEtAl},
\cite{MontesEtAl20}, and the bibliography cited therein). The symbol $[m]$
denotes a set of \emph{candidates}, or alternatives, and $\mathcal{V}%
^{(n)}=\{v_{1},\ldots,v_{n}\}$ is a set of \emph{voters}. It is assumed that
the individual preferences of the voter $v_{l}$, for $l=1,\ldots,n$, gives
rise to a \emph{linear preference ranking,} i.e. those preferences are
\emph{complete},\emph{\ transitive,} and indifference is not allowed between
any two candidates. Thus, for $l=1,\ldots,n$,
the preference ranking of the voter $v_{l}$ triggers a permutation
\emph{$r_{l}$ }over the set $[m]$. Each voter $v_{i}$ is furthermore supposed to cast her/his own vote in any
possible election, whatever could be the set $A$ ($A\in\widehat{\mathcal{P}%
}(m)$) formed by the candidates participating in that specific election.
Furthermore, the voter $v_{l}$ is supposed to cast a single vote, addressed to
the candidate within $A$, who is the preferred one according to her/his own
linear preference ranking $r_{l}$.

For $h\in\lbrack m]$ and for an ordered list $(j_{1},\ldots,j_{h})$ of
candidates (i.e. $(j_{1},\ldots,j_{h})\in\mathcal{D}(\emptyset,h)$ in the
notation used above) denote by $N_{h}^{(m)}(j_{1},\ldots,j_{h})$ the number of
all the voters who rank $j_{1},\ldots,j_{h}$ in the positions
$1,\ldots,h$, respectively. I.e., concerning preferences of those voters,
$j_{1},\ldots,j_{h}$ are the $h$ most preferred candidates, listed in the
preference order. In particular $N^{(m)}(j_{1},\ldots,j_{m})\equiv$
$N_{m}^{(m)}(j_{1},\ldots,j_{m})$ denotes the number of voters $v_{l}$'s who
share the same linear preference ranking $\left(  j_{1},\ldots,j_{m}\right)
$. The  total number of voters is then given by
\[
n=\sum_{(j_{1},\ldots,j_{m})\in\Pi_{m}.}N^{(m)}(j_{1},\ldots,j_{m}),
\]
and the set of numbers $\mathcal{N}^{(m)}=(N^{(m)}(j_{1},\ldots,j_{m}))$ for
$(j_{1},\ldots,j_{m})\in\Pi_{m}$ is typically referred to as the \emph{voting
situation. }

In an election where $A\subseteq$ $\left[  m\right]  $ is the set of
candidates, let $n_{i}(A)$ denote the total number of votes obtained by the
candidate $i\in A$ according to the afore-mentioned scenario. Thus $\sum_{i\in
A}n_{i}(A)=n$.  Similarly to Proposition \ref{prop1} we can claim that, for
$A\in\widehat{\mathcal{P}}(m)$ and $i\in A$, the numbers $n_{i}(A)$ are
determined by the voting situation $\mathcal{N}^{(m)}$.

As well-known (see e.g. the references cited above) different voting
procedures can be used for aggregating individual preferences and determining the winner
of an election. Here we assume that the winner of the election is chosen
according to the \emph{majority rule}, namely it is a candidate $j\in A$ such
that $n_{j}(A)=\max_{i\in A}n_{i}(A)$.

As a key-point for our discussion, notice that a voting situation
$\mathcal{N}^{(m)}$ gives rise to a probability distribution over $\Pi_{m}$,
by setting
\begin{equation}
\rho(j_{1},\ldots,j_{m})=\frac{N^{(m)}(j_{1},\ldots,j_{m})}{n}. \label{DistrProbAssocAVotSit}
\end{equation}

The voting-theory scenario described so far furthermore gives rise to a
ranking pattern $\boldsymbol{\tau}\in\Sigma^{(m)}$ in a way analogous to the
ranking pattern associated to the $m$-tuple $\left(  X_{1},...,X_{m}\right)
$, according to the definitions introduced in Section \ref{sec2}. In
particular, for any $A\in\widehat{\mathcal{P}}(m)$, and $i,j\in A$ with $i\neq
j$, such ranking pattern $\boldsymbol{\tau}$ satisfies the conditions
\begin{equation}
\tau(A,i)<\tau(A,j)\Leftrightarrow n_{i}(A)>n_{j}(A),\tau(A,i)=\tau
(A,j)\Leftrightarrow n_{i}(A)=n_{j}(A).
\end{equation}

We say that $\boldsymbol{\tau}$ $\ $is a ranking pattern $N$%
-\textit{concordant} with the voting situation $\mathcal{N}^{(m)}$. Notice
that $\boldsymbol{\tau}$ indicates, for any subset of candidates
$A\in\widehat{\mathcal{P}}(m)$, the one who obtains the plurality support
among the members of $A$. In particular $\boldsymbol{\tau}$ determines the
majority graph $\left(  \left[  m\right]  ,E\right)  $ which indicates the
winner for any possible direct match for pairs of candidates.

In order to demonstrate the application of Theorem \ref{EpsilonEspliciti}, to
this setting, fix now an arbitrarily given (non-weak) ranking pattern
$\boldsymbol{\sigma}$. By taking into account the position in
\eqref{DistrProbAssocAVotSit}, it is easily seen that the following two
conditions are equivalent:

i)\ There exists a voting situation $\mathcal{N}^{(m)}$ such that
$\boldsymbol{\sigma}$ is $N$-concordant with $\mathcal{N}^{(m)}$

ii)\ There exists a $m$-tuple $(X_{1},\ldots,X_{m})$ such that
$\boldsymbol{\sigma}$ is $p$-concordant with $(X_{1},\ldots,X_{m})$ and the
corresponding numbers $p^{(m)}(j_{1},\ldots,j_{m})$ ($(j_{1},\ldots,j_{m}%
)\in\Pi_{m}$) are all rational numbers.

Concerning with the $LS(\boldsymbol{\varepsilon},\boldsymbol{\sigma})$ models
introduced in the previous section, we highlight that the numbers
$p^{(m)}(j_{1},\ldots,j_{m})$ corresponding to the choice \eqref{qspiega} are definitely rational. As a direct
consequence, Theorem \ref{EpsilonEspliciti} can be used to show the
existence of an $N$-concordant voting situation  $\mathcal{N}^{(m)}$ for
$\boldsymbol{\sigma}$. Such Theorem then solves a problem in the same
direction of the mentioned result by Saari. In the version considered here our
method is limited to non-weak ranking patterns $\boldsymbol{\sigma}$. However
it provides us with an explicit construction of the desired, $N$-concordant,
voting situations. The application to special situations, examples, and
problems of voting theory will be deferred to a subsequent paper.

\bigskip

Also for Theorem \ref{lemma2} some direct applications can be found in
different fields of Probability. In particular, for a fixed binary
system in the field of Reliability Theory, it can be applied to construct
load-sharing models for the components' lifetimes, giving rise to an assigned
\textit{probability signature. }Let $S$ \ be a binary system made up with
$r$ binary components and with a logical structure $\phi:\{0,1\}^{r}%
\rightarrow\{0,1\}$. It is assumed that $S$ is \textit{coherent, }i.e. the
function $\phi$ is non-decreasing and each component is \textit{relevant}. Thus $S$ can only fail in concomitance with the failure of one of its
components. Let $T_{1},...,T_{r}$ and $F_{\mathbf{T}}$ respectively denote the
components' lifetimes and the joint probability distribution of them.

Denoting  by $T_{S}$ the lifetime of $S$, consider the basic problem of
connecting to $F_{\mathbf{T}}$ and to $\phi$ \ the survival probability
$\overline{F}_{S}(t):=\mathbb{P}\left(  T_{S}>t\right)  $.  Assuming the
\textit{no-tie} condition $\mathbb{P}\left(  T_{i}\neq T_{j}\right)  =1$ for
$i\neq j$ and denoting by $T_{1:r},...,T_{r:r}$ the order statistics of
$T_{1},...,T_{r}$, we can consider the events $\left(  T_{S}=T_{1:r}\right)
,...,\left(  T_{S}=T_{r:r}\right)  $. Since the latter events form a partition
$\overline{F}_{S}(t)$ can, for any given $t>0$, be decomposed by means of the
total probability formula
\begin{equation}
\overline{F}_{S}(t)=\sum_{k=1}^{n}\mathbb{P}\left(  T_{S}>t|T_{S}%
=T_{k:r}\right)  \cdot\mathit{p}_{k},\label{SignatureDecomposition}%
\end{equation}
where $\mathit{p}_{k}:=\mathbb{P}\left(  T_{S}=T_{k:r}\right)  $.  The vector
$\mathbf{p}\equiv\left( \mathit{p}_{1},...,\mathit{p}_{r}\right)  $ has been
called the \textit{probability signature }of $S$ (see e.g. \cite{NavSpiBal}). Generally $\mathbf{p}%
$\ depends on both the structure $\phi$ and on $F_{\mathbf{T}}$. Let now $\phi$ be fixed. In order to
understand the role of Theorem \ref{lemma2} in the construction of
load-sharing models giving rise to an assigned probability signature, notice that the probability signature $\mathbf{p}$ associated to the pair $(\phi,F_{\mathbf{T}})$ is simply determined by the family of  probabilities
\[
p^{(r)}(j_{1},...,j_{r})=\mathbb{P}(T_{1:r}=T_{j_{1}},...,T_{r:r}=T_{j_{r}}).
\]
corresponding to $F_{\mathbf{T}}$.

\end{document}